\documentclass[11pt]{amsart}
\usepackage[latin1]{inputenc}
\usepackage{mathrsfs}
\usepackage{mathtools}
\usepackage{amsmath, amsthm, amsfonts, amssymb}
\usepackage{enumitem}
\usepackage{tikz}
\usepackage{tikz-cd}
\usepackage[all]{xy}
\usepackage{xcolor}

\usepackage[left=2.3cm,top=3.3cm,right=2.3cm,bottom=2.7cm]{geometry}

\theoremstyle{plain}
\newtheorem{thm}{Theorem}[section]

\newtheorem{conjecture}[thm]{Conjecture}
\newtheorem{cor}[thm]{Corollary}
\newtheorem{conj}[thm]{Conjecture}
\newtheorem{lemma}[thm]{Lemma}
\newtheorem{lem}[thm]{Lemma}
\newtheorem{prop}[thm]{Proposition}

\theoremstyle{definition}

\newtheorem{nota}[thm]{Notation}

\numberwithin{equation}{section}

 \usepackage{todonotes}

\newcommand{\alb}{\mathrm{alb}}
\newcommand{\Alb}{\mathrm{Alb}}

\newcommand{\CH}{\mathrm{CH}}

\newcommand{\id}{\mathrm{id}}

\newcommand{\Pic}{\mathrm{Pic}}

\newcommand{\Aut}{\mathrm{Aut}}

\newcommand{\CHM}{\mathrm{\CHM}}

\usepackage[
urlcolor=blue,
colorlinks=true,
linkcolor=blue,
citecolor=blue,
]{hyperref}

\newcommand{\CC}{\mathbb{C}}
\newcommand{\QQ}{\mathbb{Q}}

\newcommand{\ZZ}{\mathbb{Z}}
\newcommand{\PP}{\mathbb{P}}

\newcommand{\sB}{\mathcal{B}}

\newcommand{\sF}{\mathcal{F}}
\newcommand{\sH}{\mathcal{H}}
\newcommand{\sM}{\mathcal{M}}
\newcommand{\sO}{\mathcal{O}}

\newcommand{\tB}{\widetilde{B}}

\begin{document}
	\title[On symplectic automorphisms of a surface with genus two fibration]{On symplectic automorphisms of a surface with genus two fibration and their action on $\CH_0$}
	\author{Jiabin Du}
	\author{Wenfei Liu}
	\address{Shanghai Institute for Mathematics and Interdisciplinary Sciences \\ Songhu Road 657  \\ Shanghai 200433 (China)}
	\email{jiabin.du@simis.cn}
	\address{Xiamen University  \\ School of Mathematical Sciences \\ Siming South Road 422 \\ Xiamen, Fujian 361005 (China)}
	\email{wliu@xmu.edu.cn}
	\date{\today}
		\begin{abstract}
		Let $S$ be a complex smooth projective surface with a genus two fibration, and $\Aut_s(S)$ the group of symplectic automorphisms, fixing every holomorphic 2-forms (if any) on $S$. Based on the work of Jin-Xing Cai, we observe in this paper that, if $\chi(\sO_S)\geq 5$, then $|\Aut_s(S)|\leq 2$.  Then we go on to verify, under some conditions, that $\Aut_s(S)$ acts trivially on the Albanese kernel $\CH_0(S)_{\alb}$ of the 0-th Chow group, which is predicted by a conjecture of Bloch and Beilinson. As a consequence, if an automorphism $\sigma\in \Aut(S)$ acts trivially on $H^{i,0}(S)$ for $0\leq i\leq 2$, then it also acts trivially on $\CH_0(S)_{\alb}$.
	\end{abstract}
\maketitle
\tableofcontents  

\section{Introduction}

Throughout the paper, we work over the complex numbers field $\CC$.

Let $X$ be a $n$-dimensional connected smooth projective variety  over $\CC$. When studying the automorphism group $\Aut(X)$, it is natural to look at the induced action of $\Aut(X)$ on the cohomology groups that are naturally attached to $X$, such as $H^*(X, R)$ with $R$ an abelian group or $H^*(X, \Omega_X^p)$ for $0\leq p\leq \dim X$, where $\Omega_X^p$ is the coherent sheaf of holomorphic $p$-forms on $X$. On the other hand, there is also an induced action of $\Aut(X)$ on the Chow groups $\CH_*(X)$, which is a more refined invariant than the cohomology groups. A deep conjecture of Bloch and Beilinson predicts that, roughly speaking, there is a natural decreasing filtration $F^{\bullet}\CH_*(X)_\QQ$ on the $\CH_*(X)_\QQ:=\CH_*(X)\otimes_{\ZZ}\QQ$ such that the induced action of $\Aut(X)$ on the graded pieces of $F^{\bullet}\CH_*(X)_\QQ$ is determined by the action of $\Aut(X)$ on the cohomology groups $H^*(X, \QQ)$. We refer to \cite[Section~11.2.2]{Voi03} for a precise statement of the Bloch--Beilinson conjecture.


For a smooth projective surface $S$, it is clear that $\CH_2(S)\cong\ZZ$ and $\CH_1(S)\cong\Pic(S)$. Thus the focus is on the still mysterious $0$-th Chow group $\CH_0(S)$. There is a natural filtration on $\CH_0(S)$: 
\begin{equation}\label{eq: filtration}
		0\subset \CH_0(S)_{\alb}\subset \CH_0(S)_{\hom}\subset \CH_0(S)
	\end{equation}
where  $\CH_0(S)_{\hom}$ is the kernel of the degree map $\deg\colon \CH_0(S)\to \mathbb{Z}$, and $\CH_0(S)_{\alb}$ is the kernel of the Albanese map $\alb_S\colon\CH_0(S)_{\hom}\to \Alb(S)$. A famous result of Mumford \cite{Mum69} says, if the geometric genus $p_g(S)>0$, then there is no uniform integer $d>0$ such that any $\alpha\in \CH_0(S)_{\hom}$ can be written as $\alpha=\alpha^{+}-\alpha^{-}\in\CH_0(S)$ with both of $\alpha^{+}$ and $\alpha^{-}$ effective and $\deg \alpha^{+}=\deg \alpha^{-}=d$. It is thus legitimate to call $\CH_0(S)_{\hom}$ and also the Albanese kernel $\CH_0(S)_{\alb}$ infinite dimensional. 

The filtration \eqref{eq: filtration} is supposed to be the filtration in the aforementioned Bloch--Beilinson conjecture, and as a consequence, we have
\begin{conjecture}[{\cite[Conjecture~11.19]{Voi03}}]\label{conj: generalize Bloch}
    Let $S$ and $T$ be smooth projective surfaces, $\Gamma$ a cycle of codimension 2 in $S\times T$ such that the map $[\Gamma]^*\colon H^{2,0}(T)\rightarrow H^{2,0}(S)$ vanishes. Then the map $\Gamma_*\colon \CH_0(S)_\alb\rightarrow \CH_0(T)_\alb$ vanishes.
\end{conjecture}

By taking $T=S$ with $p_g(S)=0$, and $\Gamma=\Delta_S\subset S\times S$ the diagonal in Conjecture~\ref{conj: generalize Bloch}, we recover Bloch's initial conjecture. \begin{conj}[{\cite{Blo75}, see also \cite[Conjecture~11.2]{Voi03}}]\label{conj: pg=0}
Let $S$ be a smooth projective surface. If $p_g(S)=0$ then $\CH_0(S)_{\alb}=0$. 
\end{conj}
Bloch's conjecture has been verified in various special cases (see \cite[page 444]{DL23} for a discussion), but is widely open in general.

We are interested in the induced action of automorphisms on $\CH_0(S)$. So, take $T=S$, and $\Gamma:=\Gamma_\sigma-\Delta_S$, where $\Gamma_\sigma$ is the graph of an automorphism $\sigma$ in a subgroup $G\subset \Aut(S)$, and we obtain from Conjecture~\ref{conj: generalize Bloch} the following
\begin{conjecture}[{cf.~\cite[Conjecture~1.2]{Voi12}}]\label{conj: Aut_s act on CH}
    Let $S$ be a smooth projective surface, and $G$ a group of automorphisms of $S$ acting trivially on $H^{2,0}(S)$. Then $G$ acts trivially on $\CH_0(S)_\alb$.
\end{conjecture}
Define the group $\Aut_{s}(S)$ of \emph{symplectic automorphisms}
\[
\Aut_{s}(S):=\{\sigma\in \Aut(S)\mid \text{$\sigma$ induces the trivial action on $H^{2,0}(S)=H^0(S, K_S)$}\}.
\]
and the group $\Aut_{\sO}(S)$ of \emph{$\sO$-cohomologically trivial automorphisms}
\[
\Aut_{\sO}(S):=\{\sigma\in \Aut(S)\mid \text{$\sigma$ induces the trivial action on $H^i(S, \sO_S)$ for any $i$}\}
\]
By Serre duality, we have $H^0(S, K_S)\cong H^2(S,\sO_S)^{\vee}$, and hence $\Aut_{\sO}(S)\subset \Aut_s(S)$; the two groups coincide if the irregularity $q(S):=\dim H^1(S, \sO_S)=0$. The automorphisms in $\Aut_{s}(S)$ are those fixing a general holomorphic 2-form (if any) on $S$, and hence called symplectic. 

\medskip

Conjecture~\ref{conj: Aut_s act on CH} amounts to saying that $\Aut_s(S)$ acts trivially on $\CH_0(S)_\alb$.

\medskip

For surfaces with $p_g(S)>0$, Conjecture~\ref{conj: Aut_s act on CH} is known to hold in the following cases:
\begin{itemize}
    \item $S$ is an abelian surface and $G=\Aut_{s}(S)$ (\cite[Theorems~A.1 and A.7]{BKL76} and \cite[Corollary~1.5]{Paw19});
    \item $S$ is a K3 surface, and $G\subset \Aut_{s}(S)$ is finite (\cite{Voi12, Huy12});
    \item $S=K(A)$ is the Kummer K3 surface associated to an abelian surface $A$, and $G$ is generated by $\sigma\in \Aut_s(S)$ that lifts to a group automorphism $\tilde\sigma$ of $A$ (\cite[Theorem~1.10]{Paw19});
    \item $S$ is a K3 surface with an elliptic fibration $f\colon S\rightarrow B$, and $G$ preserves the fibration structure $f$ (\cite[Theorem~1.5]{DL23});
    \item $S$ is a K3 surface with either the Picard number $\rho(S)\geq 3$ or $S$ admitting a Jacobian fibration, and $G=\Aut_s(S)$ (\cite[Theorem~1.3]{LYZ23});
    \item $S$ has Kodaira dimension one and $G=\Aut_{\sO}(S)$ (\cite[Theorem~5.10]{DL23});
    \item $S$ has Kodaira dimension one with $(p_g,q)\notin\{ (1,1),\, (2,2)\}$, and $G=\Aut_s(S)$ (\cite[Theorem~1.3]{DL23});
    \item $S$ has a finite Chow motive in the sense of Kimura (\cite{Kim05}) (this is the case if $S$ has an isotrivial fibration), and $G\subset\Aut_s(S)$ is finite (\cite[Lemma~5.9]{DL23}).
\end{itemize}

In this paper, we investigate Conjecture~\ref{conj: Aut_s act on CH} for surfaces of general type with a genus two fibration. Recall that a fibration of genus $g$ on a smooth projective surface $S$ means a morphism $f\colon S\to B$ onto a smooth projective curve $B$ with connected fibers of genus $g$. 
\begin{thm}\label{thm: main}
Let $S$ be a surface of general type with a genus two fibration $f\colon S\rightarrow B$ and $\chi(\sO_S)\geq 5$. Suppose that $\Aut_s(S)$ is nontrivial. Then the following holds.
\begin{enumerate}
\item[(i)] $|\Aut_{s}(S)|=2$.
\item[(iii)] If the canonical map $\phi_{K_S}\colon S\dashrightarrow \PP^{p_g(S)-1}$ is composed with a pencil, then $\Aut_{s}(S)$ preserves every fiber of $f$, and acts trivially on $\CH_0(S)_{\alb}$.
\item[(iv)] If the canonical map $\phi_{K_S}\colon S\dashrightarrow \PP^{p_g(S)-1}$ is generically finite onto its image $T:=\phi_{K_S}(S)$, then $\Aut_{\sO}(S)$ is trivial, and $\Aut_{s}(S)$ acts trivially on $\CH_0(S)_\alb$ unless 
\begin{enumerate}
    \item $q(S)=g(B)\geq \chi(\sO_S)-3$,
    \item a smooth model of the quotient surface $S/\langle\sigma\tau\rangle$ is of general type with $p_g=q=0$, where $\sigma$ is the generator of $\Aut_s(S)$ and $\tau$ is the hyperelliptic fibration of $f$, and
    \item $f$ is not isotrivial.
\end{enumerate}
\end{enumerate}
\end{thm}


The bound $|\Aut_{s}(S)|\leq 2$ of Theorem~\ref{thm: main} (i) is implicit in Jin-Xing Cai's work \cite{Cai06a, Cai06b} on the automorphism group $H$ of genus two fibrations acting trivially on $H^2(S, \QQ)$. More specifically, by Hodge decomposition, one has $H\subset \Aut_s(S)$, and his proof that $|H|\leq 2$ uses only this fact. 

Cai also constructed various fibered surfaces of genus two with an involution acting trivially on $H^2(S, \QQ)$ in \cite{Cai06a, Cai06b, Cai07}. Only one series of them, namely \cite[Example~3.5]{Cai06b}, satisfies all of the conditions (a)--(c) of Theorem~\ref{thm: main}, and we do not know whether or not $\Aut_s(S)$ acts trivially on $\CH_0(S)_\alb$ for $S$ in this example.

Theorem~\ref{thm: main} follows from Propositions~\ref{prop: canonical map fibration}, Proposition~\ref{prop: canonical map gen finite}, and Corollary~\ref{cor: isotrivial}. Let us explain the ideas of the proofs. Following \cite{Cai06a, Cai06b}, the proof of Theorem~\ref{thm: main} (i) is based on the results of Xiao on genus two fibrations (\cite{Xiao85}). Specifically, let $G\subset \Aut(S)$ be the subgroup generated by $\Aut_s(S)$ together with the hyperelliptic involution $\tau$. Then the canonical map $\phi_{K_S}$ of $S$ factors through the quotient map $S\rightarrow S/G$ (Lemma~\ref{lem: factor thru S/G}), and the explicit bounds for $\phi_{K_S}$ of Xiao (Theorem~\ref{thm: Xiao canonical map}) give the bound $|G|\leq 4$ and hence $|\Aut_s(S)|\leq 2$. 

Once the bound on $|\Aut_s(S)|$ is established, we have 
\[
G=\{ \id_S,\,\sigma,\,\tau,\,\sigma \tau\}\cong (\ZZ/2\ZZ)^2 .
\]
We can then decompose $\CH_0(S)_{\alb, \QQ}$ into eigenspaces with respect to the $G$-action, and see that $\Aut_s(S)$ acts trivially on $\CH_0(S)_{\alb}$ as soon as $\CH_0(S/\langle\sigma\tau\rangle)_{\alb} = 0$ (Lemma~\ref{lem: Z22 acts}). Since $p_g(S/\langle\sigma\tau\rangle)=0$, the Albanese kernel $\CH_0(S/\langle\sigma\tau\rangle)_{\alb}$ vanishes if   (a smooth model of) the quotient surface $S/\langle\sigma\tau\rangle$ is not of general type by \cite{BKL76}. The bulk of our arguments is devoted to verify the latter condition, once the surface $f\colon S\rightarrow B$ is not isotrivial but violates one of the conditions (a) and (b) in Theorem~\ref{thm: main}. The case of isotrivial fibrations has been settled once and for all by applying Kimura's finite dimensionality of such surfaces (Corollary~\ref{cor: isotrivial}).

Assuming Conjecture~\ref{conj: pg=0} holds, then $\CH_0(S/\langle\sigma\tau\rangle)_{\alb} = 0$, and we are done again (Proposition~\ref{prop: assume conj}).

Concerning the condition $\chi(\sO_S)\geq 5$ in the theorem, we remark that minimal surfaces of general type with $\chi(\sO_S)<5$ form (only) a bounded family. This condition is used mainly to ensure that the canonical map of $S$ is well-behaved and that $S$ has at most one genus two fibration on it. 

Finally, as a consequence of Theorem~\ref{thm: main}, Conjecture~\ref{conj: Aut_s act on CH} holds for the subgroup $\Aut_\sO(S)\subset \Aut_s(S)$ for surfaces of general type with a genus two fibration, whose invariants are not so small.
\begin{cor}\label{cor: main}
Let $S$ be a surface of general type with a genus two fibration $f\colon S\rightarrow B$ and $\chi(\sO_S)\geq 5$. Then $\Aut_{\sO}(S)$ acts trivially on $\CH_0(S)_\alb$.
\end{cor}

The organization of the paper is as follows: We recall in Section~\ref{sec: prelim} some relevant notions and facts such as fibration-preserving automorphisms, the induced action on the Albanese variety and the $0$-th Chow group, as well as useful criteria for the triviality of the induced action on the Albanese kernel. In Section~\ref{sec: g=2}, we focus on genus two fibrations, first using the canonical map to bound the symplectic automorphism group $\Aut_s(S)$, and then verifying the triviality of its induced action on the Albanese kernel as stated in Theorem~\ref{thm: main}.

\medskip

\noindent {\bf Notation and Conventions.}

\medskip

Let $S$ be smooth projective surface over $\CC$.
\begin{itemize}
\item For a coherent sheaf $\sF$ on $S$,  we denote $h^i(X, \sF):=\dim_\CC H^i(X, \sF)$, and  $\sF^{\vee}:=\sH om_{\sO_S}(\sF, \sO_S)$, the dual of $\sF$.
\item For $0\leq i\leq 2$, $\Omega_S^i$ denotes the sheaf of $i$-forms on $S$.
\item $\omega_S $ denotes the canonical sheaf of $S$, which can be identified with $\Omega_S^2$, and $K_S$ denotes a canonical divisor of $S$. 
\item The following numerical invariants are attached to $S$:
\begin{itemize}
    \item the \emph{geometric genus} $p_g(S):=h^0(S, \omega_S)$;
    \item the \emph{irregularity} $q(S):=h^1(S, \sO_S) = h^0(S, \Omega_S^1)$;
    \item the \emph{holomorphic Euler characteristic} $\chi(\sO_S) = 1- q(S) + p_g(S)$.
\end{itemize}
\end{itemize} 
For a singular projective surface $T$, say with rational singularities, its geometric genus $p_g(T)$, irregularity $q(T)$ are defined as the corresponding invariants of its smooth model.

For a finite group $G$ and an element $\sigma\in G$, their orders are denoted by $|G|$ and $|\sigma|$ respectively. An action of $G$ on a set $X$ is called \emph{trivial} if $\sigma(x)=x$ for any $\sigma\in G$ and any $x\in X$.

\medskip

\noindent{\bf Acknowledgements.} We would like to thank Professor Jin-Xing Cai for communications on some questions related to the paper.


\section{Preliminaries}\label{sec: prelim}
In this section, we recall the notions and facts that are used in this paper.
\subsection{Fibered surfaces and their automorphisms}\label{key}

Let $S$ be a normal projective surface, and $f\colon S\to B$ a fibration onto a smooth projective curve, that is, $f$ is a surjective morphism with connected fibers. The fibration is called 
\begin{itemize}
    \item \emph{of genus $g$} if its general fiber has genus $g$;
    \item \emph{hyperelliptic} (resp.~\emph{elliptic}, resp.~\emph{a $\PP^1$-fibration}) if its smooth fibers are hyperelliptic (resp.~elliptic, resp.~$\PP^1$);
        \item \emph{isotrivial} if its smooth fibers are mutually isomorphic.
    \end{itemize} We call $q_f:=q(S)-g(B)$ the \emph{relative irregularity} of $f$.


The subgroup $\Aut_{f}(S)$ of \emph{fibration-preserving automorphisms} is defined as follows:
  \[
    \Aut_{f}(S):= \{\sigma\in \Aut(S) \mid \text{$\sigma$ maps fibers of $f$ to fibers} \}
    \]
    There is an induced action of $\Aut_f(S)$ on $B$, that is, a homomorphism $r\colon \Aut_f(S)\rightarrow \Aut(B)$, such that for any $\sigma\in \Aut_f(S)$ there is an induced automorphism $\sigma_B=r(\sigma)\in \Aut(B)$ such that the following diagram is commutative:
    \[
    \begin{tikzcd}
        X \arrow[r, "\sigma"]\arrow[d, "f"']&  X \arrow[d, "f"]\\
        B \arrow[r, "\sigma_B"] & B
    \end{tikzcd}
    \]
The elements of $\Aut_B(S):=\ker r$ are called \emph{fiber-preserving automorphisms}.

\begin{lemma}\label{lem-finitenessAut}
    Let $f\colon S\to B$ be a non-isotrivial fibration of genus $g\geq 2$. Then $\Aut_f(S)$ is a finite group.
\end{lemma}
\begin{proof}
    Consider the  moduli map associated to $f$
    $$
    \mu\colon B\to \overline{\sM}_g,
    $$
    where $\overline\sM_g$ is the (compact) moduli space of stable curves of genus $g$. Since $f$ is not isotrivial, $\mu$ is generically finite, and the image $\Aut_f(S)|_B$ of $r\colon\Aut_f(S)\rightarrow \Aut(B)$ has order at most $\deg\mu$. It follows that
    \begin{equation}\label{eq: bound Aut_f}
        |\Aut_f(S)|= |\Aut_B(S)|\cdot |\Aut_f(S)|_B|= |\Aut(F)|\cdot |\Aut_f(S)|_B| \leq |\Aut(F)|\cdot \deg\mu
    \end{equation}
where $F$ is a general fiber of $f$. Since $g(F)=g\geq 2$, $\Aut(F)$ is a finite group, and hence $\Aut_f(S)$ is also a finite group by \eqref{eq: bound Aut_f}.
\end{proof}

\subsection{The induced action on the Albanese variety and the $0$-th Chow group}
The Chow groups of a normal projective variety $X$ is the group of rational equivalence classes of algebraic cycles on $X$; we refer to \cite[Chapter~9]{Voi03} for the basic properties of Chow groups.

Let $X$ be a smooth projective variety. Fixing a base point $x_0\in X$, we may define a group homomorphism
\[
\alb\colon\CH_0(X)_{\hom}\rightarrow \Alb(X) = H^0(X, \Omega_{X}^1)^{\vee}/H_1(X, \ZZ),\quad \sum_i n_i [x_i] \mapsto \sum_{i} n_i\left[\int_{x_0}^{x_i}\right]
\]
where $\int_{x_0}^{x_i}\colon H^0(X, \Omega_{X}^1)\rightarrow \CC$ maps a holomorphic 1-form $\eta$ to the integral $\int_{x_0}^{x_i}\eta$, determined up to $\int_\gamma \eta$ for a closed loop $\gamma$ on $X$. The homomorphism $\alb$ does not depend on the choice of the base point $x_0$.

Any automorphism group $G\subset \Aut(X)$ has an induced action on $\CH_0(X)$ and $\Alb(X)$ as follows: For $\sigma\in G$, $\sum_i n_i[x_i]\in \CH_0(X)$, and $\left[\int_{x_0}^x\right]\in \Alb(X)$, 
\[
\sigma_*(\sum_i n_i[x_i])= \sum_i n_i[\sigma(x_i)], \quad \sigma_*(\left[\int_{x_0}^x\right])= \left[\int_{\sigma(x_0)}^{\sigma(x)}\right]=\left[\int_{x_0}^{\sigma(x)}\right]-\left[\int_{x_0}^{\sigma(x_0)}\right]
\]
Note that $G$ acts by group automorphisms on $\CH_0(X)$ and $\Alb(X)$, and the homomorphism $\alb$ is $G$-equivariant. The $G$-action extends in an obvious way to $\CH_0(X)_\QQ:=\CH_0(X)\otimes_\ZZ \QQ$ and $\Alb(X)_\QQ:=\Alb(X)\otimes_{\ZZ}\QQ$.

Let $Z$ be a normal projective birational model of $X$, with at most rational singularities. Then we have natural identifications (see \cite[Section~2.1]{JLZ23} for a discussion for the Albanese variety of a singular variety):
\[
\Alb(Z) \cong \Alb(X),\quad \CH_0(Z) \cong \CH_0(X).
\]

The following basic but important fact about the $G$-actions on $\Alb(X)_\QQ$ and $\CH_0(X)_\QQ$ should be well-known, but we write a proof of it for completeness (cf.~\cite[Lemma~1.6]{Lat21}). 
\begin{lemma}\label{lem: CH inv}
 Let $X$ be a smooth projective variety and $G\subset \Aut(X)$ a finite subgroup. Let $Y$ be a normal projective birational model with rational singularities of the quotient surface $X/G$. Then 
 $$
 \Alb(X)_{\QQ}^G\cong \Alb(Y)_{\QQ},\quad \CH_0(X)_{\alb,\QQ}^{G}\cong \CH_0(Y)_{\alb,\QQ},\quad \CH_0(X)_{\hom,\QQ}^{G}\cong \CH_0(Y)_{\hom,\QQ}.
 $$ 
\end{lemma}
\begin{proof}
Since $X/G$ and $Y$ have at most rational singularities, there are natural identifications
\[
\Alb(X/G)=\Alb(Y), \quad \CH_0(X/G)_{\hom, \QQ}\cong \CH_0(Y)_{\hom, \QQ},\quad\CH_0(X/G)_{\alb,\QQ}\cong \CH_0(Y)_{\alb,\QQ}.
\]
Thus we may assume that $Y=X/G$.

Recall that 
\[
\Alb(X)_\QQ = \frac{H^0(X, \Omega_X^1)^{\vee}\otimes_\ZZ\QQ}{H_1(X, \QQ)},
\]
and hence we have natural isomorphisms of $\QQ$-vector spaces
\begin{equation}\label{eq: Alb^G}
    \Alb(X)_\QQ^G \cong  \frac{\left(H^0(X, \Omega_X^1)^G\right)^{\vee}\otimes_\ZZ \QQ}{H_1(X, \QQ)^G}\cong  \frac{H^0\left(Y, \Omega_Y^1\right)^{\vee}\otimes_\ZZ \QQ}{H_1(Y, \QQ)} \cong \Alb(X/G)_\QQ
\end{equation}
where $Y$ is a smooth projective model of $X$.

By the universal property of the Albanese morphism and by the flatness of $\QQ$ as a $\ZZ$-module, the quotient map $\pi\colon X\to X/G$ induces the following commutative diagram with exact rows:
\[
\begin{tikzcd}
 0\arrow[r]&\CH_0(X)_{\alb,\QQ}\arrow[r]\arrow[d, "\pi_*"]&\CH_0(X,\QQ)_{\hom,\QQ}\arrow[r]\arrow[d, "\pi_*"]&\Alb(X)_{\QQ}\arrow[r]\arrow[d, "\pi_*"]&0\\
0\arrow[r]&\CH_0(X/G)_{\alb,\QQ}\arrow[r]&\CH_0(X/G)_{\hom,\QQ}\arrow[r]&\Alb(X/G)_\QQ\arrow[r]&0   
\end{tikzcd}
\]
Restricting to the $G$-invariant parts of the $\QQ$-vector spaces in the first row, we obtain
the following commutative diagram with exact rows:
\[
\begin{tikzcd}
 0\arrow[r]&\CH_0(X)_{\alb,\QQ}^G\arrow[r]\arrow[d, "\alpha"]&\CH_0(X,\QQ)_{\hom,\QQ}^G\arrow[r]\arrow[d, "\beta"]&\Alb(X)_{\QQ}^G\arrow[r]\arrow[d, "\gamma"]&0\\
0\arrow[r]&\CH_0(X/G)_{\alb,\QQ}\arrow[r]&\CH_0(X/G)_{\hom,\QQ}\arrow[r]&\Alb(X/G)_\QQ\arrow[r]&0   
\end{tikzcd}
\]
and we have seen in \eqref{eq: Alb^G} that $\gamma$ is an isomorphism.

For a point $x\in X$, let $\bar x$ be its image in $X/G$. Any $0$-cycle on $X/G$ is rationally equivalent to one with support outside the branch locus $\sB$ of $\pi$ (\cite[Fact~3.3]{Voi12}).  For $[z]\in \CH_0(X/G)_{\hom,\QQ}$, we may thus assume that $z=\sum_i n_i \bar x_i$ with $\bar x_i\not\in \sB$ for any $i$, and one sees that $\beta$ has an inverse
\[
\beta^{-1}([z]) = \frac{1}{|G|}\sum_i \sum_{g\in G} [g(x_i)].
\]
Therefore, $\beta$ is an isomorphism. 

By the Five Lemma, $\alpha$ is also an isomorphism.
\end{proof}

\subsection{Useful criteria for an symplectic automorphism to act trivially on $\CH_0(S)_{\alb}$}
\begin{lemma}[{cf.~\cite[Lemma~5.9]{DL23}}]\label{lem:fdChow}
	Let $S$ be a smooth projective surface. Assume that the Chow motive $\mathfrak{h}(S)$  of $S$ is finite dimensional in the sense of  Kimura \cite{Kim05} and O'Sullivan. Then any symplectic automorphism of finite order of $S$ acts trivially on $\CH_0(S)_{\alb}$.
\end{lemma}
In particular, one has 
\begin{cor}\label{cor: isotrivial}
	Let $f\colon S\to B$ be an isotrivially fibered surface. Then any finite order symplectic automorphism of $S$ acts trivially on $\CH_{0}(S)_{\alb}$.
\end{cor}
\begin{proof}
Since $f$ is an isotrivial fibration, $S$ is birational to $(\tB\times F)/G$, where $\tB$ is a smooth projective curve dominating $B$ and $F$ is a general fiber of $f$. Therefore, $S$ is dominated by $\tB\times F$ and hence has finite dimensional Chow motive by \cite{Kim05}. Now apply Lemma~\ref{lem:fdChow}.
\end{proof}

\begin{cor}\label{cor: q_f max}
	Let $f\colon S\to B$ be a fibered surface of genus $g$ such that $q_f=g$. Then any finite order symplectic automorphism of $S$ acts trivially on $\CH_{0}(S)_{\alb}$.
\end{cor}
\begin{proof}
Since $g=q_f$, the fibration $f$ is isotrivial by \cite{Bea82}, and hence the assertion follows from Corollary~\ref{cor: isotrivial}.
\end{proof}

If a surface $S$ has several commuting involutions, then we have the following criterion for the triviality of the $\Aut_s(S)$-action on the Albanese kernel.
\begin{lemma}\label{lem: Z22 acts}
Let S be a smooth projective surface. Let $G=\langle \sigma,\tau\rangle\cong (\ZZ/2\ZZ)^2$ be a subgroup of $\Aut(S)$. Suppose that $\CH_0(S/\langle\tau\rangle)_{\alb}=0$. Then $\sigma$ acts trivially on $\CH_0(S)_{\alb}$ if and only if $\CH_0(S/\langle\sigma\tau\rangle)_{\alb}=0$.
\end{lemma}
\begin{proof}
Note that $\CH_0(S)_{\alb}$ has no torsion (\cite{Roj80}), we may work on the Albanese kernels with $\QQ$-coefficients. Since $\CH_0(S/\langle\tau\rangle)_{\alb}=0$, $\tau$ acts as $-1$ on $\CH_0(S/\langle\tau\rangle)_{\alb,\QQ}$. Thus $\sigma$ acts as identity on $\CH_0(S)_{\alb,\QQ}$ if and only if $\sigma\tau$ acts as $-1$ on $\CH_0(S)_{\alb,\QQ}$, which is equivalent to the vanishing of the $(\sigma\tau)$-invariant part $\CH_0(S)_{\alb, \QQ}^G \cong \CH_0(S/\langle\sigma\tau\rangle)_{\alb,\QQ}$.
\end{proof}

If a surface  with vanishing geometric genus is not of general type, then it has trivial Albanese kernel (\cite{BKL76}). Thus we may obtain from Lemma~\ref{lem: Z22 acts} the following statement.
\begin{cor}[{\cite[Lemma~5.1]{DL23}}]\label{cor: Z22 acts}
Let S be a smooth projective surface. Let $G=\langle \sigma,\tau\rangle\cong (\ZZ/2\ZZ)^2$ be a subgroup of $\Aut(S)$ such that (the smooth models of) the quotient surfaces $S/\langle\tau\rangle$ and $S/\langle\sigma\tau\rangle$ are not of general type and have vanishing geometric genus. Then $\sigma$ acts trivially on $\CH_0(S)_{\alb}$.
\end{cor}

\section{Symplectic automorphisms of surfaces with a genus two fibration}\label{sec: g=2}


For a smooth projective surface $S$ with $p_g(S)>0$, one may use the canonical system $|K_S|$ to define a rational map $\phi_{K_S}\colon S\dashrightarrow\PP^{p_g(S)-1}$, called the \emph{canonical map} of $S$. This map is a main tool in our study of symplectic automorphisms of surfaces with a genus two fibration, due to the following observation.
\begin{lem}\label{lem: factor thru S/G}
Let $S$ be a smooth projective surface of general type with a hyperelliptic fibration $f\colon S\rightarrow B$, and $G\subset \Aut(S)$ the subgroup generated by the hyperelliptic involution $\tau\in \Aut_B(S)$ together with the symplectic automorphism group $\Aut_s(S)$. Suppose that $p_g:=p_g(S)>0$, so that the canonical map $\phi_{K_S}\colon S\dashrightarrow \PP^{p_g-1}$ is well-defined. Then $\phi_{K_S}$ factors through the quotient map $\pi\colon S\rightarrow S/G$, that is, there is a rational map $\varphi\colon S/G\dashrightarrow  \PP^{p_g-1}$ such that $\phi_{K_S} = \varphi\circ\pi$.
\end{lem}
\begin{proof}
Since $\tau$ acts as $-1$ on $H^0(S, K_S)$ while $\Aut_s(S)$ acts trivially on $H^0(S, K_S)$, $\phi_{K_S}$ is a $G$-equivariant map with $G$ acting trivially on the target. The assertion of the lemma follows.
\end{proof}

Before recalling the fundamental results of Xiao on the canonical map of a surface of general type with a genus two fibration, we need to introduce some notation first. 
\begin{nota}\label{nota: g=2 fibration}
Let $S$ be a surface of general type, and $f\colon S\rightarrow B$ a genus two fibration. We abbreviate the Euler characteristic $\chi(\sO_S)$, the geometric genus $p_g(S)$, the irregulaity $q(S)$, the genus $g(B)$ of the base curve $B$ by $\chi,\, p_g,\,q,\, b$ respectively. Let $M\subset f_*\omega_S$ be  an invertible subsheaf with maximal degree and  $e:=2\deg M - \deg f_*\omega_S$. Then we have (\cite[Proof of Theorem~2.1]{Xiao85}) 
\[
\deg f_*\omega_S = \chi + 3(b-1),
\]
and hence 
\begin{equation}\label{eq: deg M}
\deg M = \frac{1}{2}(e+\deg f_*\omega_S) = \frac{1}{2} (e +\chi + 3(b-1) ).
\end{equation}
As in Lemma~\ref{lem: factor thru S/G}, we will denote by $\tau$ the hyperelliptic involution of $f$, and let $G\subset \Aut(S)$ be the subgroup generated by $\Aut_s(S)$ and $\tau$. Note that $\tau$ acts as $-1$ on $H^0(S, K_S)$ and hence $\tau\notin\Aut_s(S)$ as soon as $p_g(S)>0$.
\end{nota}
\begin{thm}[{\cite[pages 71--73, Corollaire 1 and Proposition~5.2]{Xiao85}}]\label{thm: Xiao canonical map}
 Let $S$ be a smooth projective surface of general type with a genus two fibration $f\colon S\rightarrow B$, as in Notation~\ref{nota: g=2 fibration}. 
 \begin{enumerate}
     \item[(i)] If $p_g\geq 3$ and  the image of the canonical map $\phi_{K_S}$ is a curve, then $\phi_{K_S}$ factors through $f$, and there are three possibilities for the numerical invariants:
\begin{enumerate}
\item $q=b=0$ and $e=p_g$;
\item $b=0, q=1$ and $e=p_g+1$;
\item $q=b=1$ and $e=p_g$.
\end{enumerate}
\item[(ii)] If $\chi(\sO_S)\geq 4$ and  the image of the canonical map is a surface, then $\deg\phi_{K_S}\in\{2,4\}$. Moreover, if $\deg\phi_{K_S}=4$, then
\begin{enumerate}
\item 
$p_g(S)\leq 2b+2$;
\item $T:=\phi_{K_S}(S)$ is either a rational surface or a cone over an elliptic curve.
\end{enumerate}
 \end{enumerate}
\end{thm}

Now we discuss $\Aut_s(S)$ for a fibered surface $f\colon S\rightarrow B$ of genus two according to the behavior of the canonical map, given by Theorem~\ref{thm: Xiao canonical map}.

\begin{prop}\label{prop: canonical map fibration}
Let $S$ be a smooth projective surface of general type with a genus two fibration $f\colon S\rightarrow B$ such that $p_g(S)\geq 3$. Suppose that the image of the canonical map $\phi_{K_S}$ is a curve. Then the following holds.
\begin{enumerate}
    \item[(i)] $\Aut_s(S)\subset\Aut_B(S)$, that is, each symplectic automorphism of $S$ is fiber-preserving.
    \item[(ii)] $\Aut_s(S)$ has order at most $2$, and it acts trivially on $\CH_0(S)_{\alb}$.
\end{enumerate} 
\end{prop}
\begin{proof}
By Theorem~\ref{thm: Xiao canonical map} (i), $\phi_{K_S}$ factors through $f$. Therefore, the fibration $f$ is canonical, and every automorphism of $S$ preserves it. In other words, we have $\Aut(S)=\Aut_f(S)$.

By \eqref{eq: deg M}, we have
\[
\deg M = 
\begin{cases}
p_g - 1 & \text{ in cases (a) and (b) of Theorem~\ref{thm: Xiao canonical map} (i),} \\
p_g  & \text{ in case (c) of Theorem~\ref{thm: Xiao canonical map} (i).} 
\end{cases}
\]
In all cases, $\deg M\geq 2b +1$ and hence $M$ is a very ample invertible sheaf on $B$. Moreover, by counting dimensions, we have $H^0(B, M)\cong H^0(B, f_*\omega_S)\cong H^0(S, \omega_S)$, and hence the canonical map $\phi_{K_S}$ factors as
\[
\begin{tikzcd}
\phi_{K_S}\colon S\arrow[r, "f"] &  B \arrow[hook, r, "\phi_M"] &  \PP^{p_g-1}
\end{tikzcd}
\]
where $\phi_M$ is the embedding defined by the complete linear system $|M|$.

\medskip

Note that the morphisms $\phi_{K_S}$, $f$, and $\phi_M$ are all $G$-equivariant. Since $\Aut_s(S)$ acts trivially on $\PP^{p_g-1}$ and $B$ embeds into  $\PP^{p_g-1}$, the induced action of  $\Aut_s(S)$  on $B$ is trivial, and hence $\Aut_s(S) \subset \Aut_B(S)$. This proves (i).

\medskip

(ii) follows from (i) and the following Lemma~\ref{lem: AutB(S)}.
\end{proof}

\begin{lem}\label{lem: AutB(S)}
Let $S$ be a smooth projective surface, and $f\colon S\to B$ a fibration of genus two. If $p_g(S)>0$, then $\Aut_B(S)\cap \Aut_s(S)$ has order at most $2$, and it acts trivially on $\CH_0(S)_{\alb}$.
\end{lem}
\begin{proof}
Let $F$ be a general fiber of $f$. Then $\Aut_B(S)$ injects into $\Aut(F)$, which is finite. Suppose that $\sigma\in \Aut_B(S)\cap \Aut_s(S)\setminus\{\id_S\}$. Then 
\[
p_g(S)=p_g(S/\langle\sigma\rangle)>0.
\]
This implies that $g(F/\langle\sigma\rangle)=1$, since it is smaller than $2$ but cannot be $0$, due to the positivity of $p_g$. By the Riemann-Hurwitz formula, we have 
	\begin{equation}\label{eq:RH}
		2g(F)-2=|\sigma|\left(2g(F/\langle\sigma\rangle)-2+\sum_i(1-\frac{1}{r_i})\right).
	\end{equation}	
	Since the automorphism group $\langle\sigma\rangle$ generated by $\sigma$ is abelian, the quotient map $F\to F/\langle\sigma\rangle$ has at least two branch points and hence $\sum_i(1-\frac{1}{r_i})\geq 1$. Then the equality (\ref{eq:RH}) implies that $|\sigma|\leq 2$. 
    
Note that the hyperelliptic involution $\tau\in \Aut_B(S)$ commutes with $\sigma$, and they generate a Klein group $G:=\langle \sigma, \tau\rangle\cong(\ZZ/2\ZZ)^2$. Since the induced fibrations $S/\langle\sigma\tau\rangle\rightarrow B$ and $S/\langle \tau\rangle\rightarrow B$ have genus $1$ and $0$ respectively, so both surfaces are not of general type, and they have vanishing $p_g$, we may conclude by Corollary~\ref{cor: Z22 acts} that $\sigma$ acts trivially on $\CH_0(S)_{\alb}$.
\end{proof}

\begin{lem}\label{lem: Aut=Aut_f}
    Let $S$ be a smooth projective surface of general type with a genus two fibration $f\colon S\rightarrow B$. If $\chi(\sO_S)\geq 5$, then $\Aut(S)=\Aut_f(S)$, that is, every automorphism of $S$ preserves the fibration $f$.
\end{lem}
\begin{proof}
This is because $f$ is the unique genus two fibration on $S$ by \cite[Proposition~6.4 and Th\'eor\`em~6.5]{Xiao85}, and hence preserved by any automorphism of $S$. 
\end{proof}

\begin{prop}\label{prop: canonical map gen finite}
Let $S$ be a smooth projective surface of general type with a genus two fibration $f\colon S\rightarrow B$, as in Notation~\ref{nota: g=2 fibration}. Suppose that  $\chi(\sO_S)\geq 5$ and the canonical map $\phi_{K_S}$ is generically finite onto its image $T:=\phi_{K_S}(S)$. If $\Aut_s(S)$ is nontrivial, then the following holds.
\begin{enumerate}
\item[(i)] $\Aut_s(S)$ has order 2, and preserves the fibration $f$.
\item[(ii)] $S\dashrightarrow T$ is birationally a $(\ZZ/2\ZZ)^2$-cover, that is, the induced extension $K(S)/K(T)$ of function fields is Galois with Galois group isomorphic to $(\ZZ/2\ZZ)^2$.
\item[(iii)] $b\geq 2$, and the induced action of $\Aut_s(S)$ on $B$ is not trivial.
\item[(iv)] $\Aut_{\sO}(S)$ is trivial.
\item[(v)] if $q_f>0$, then $\Aut_s(S)$ acts trivially on $\CH_0(S)_\alb$.
\end{enumerate}
\end{prop}
\begin{proof}
By Lemma~\ref{lem: factor thru S/G}, $\phi_{K_S}$ factors through the quotient map $S\rightarrow S/G$, where $G\subset \Aut(S)$ is the subgroup generated by $\Aut_s(S)$ and the hyperelliptic involution $\tau$, which does not lie in $\Aut_s(S)$. Therefore, 
\[
2|\Aut_s(S)| \leq |G|\leq \deg \phi_{K_S}\in \{2,4\}
\]
where we used the bound on $\deg \phi_{K_S}$ from Theorem~\ref{thm: Xiao canonical map} (ii).
Since $\Aut_s(S)$ is nontrivial by assumption, we infer that
\[
\deg\phi_{K_S} = 4, \quad |\Aut_s(S)|=2, \quad G\cong (\ZZ/2\ZZ)^2.
\]
By Lemma~\ref{lem: Aut=Aut_f}, $\Aut_s(S)$ preserves the fibration $f$. Thus (i) is proved.

\medskip

Note that $\phi_{K_S}$ has degree 4 and factors through the quotient map $S\rightarrow S/\langle \tau, \sigma\rangle$, which also has degree $4$. Thus the induced map $S/\langle \tau, \sigma\rangle\dashrightarrow T$ is birational, and (ii) follows.

\medskip

By Theorem~\ref{thm: Xiao canonical map} (ii), we have
\begin{equation}\label{eq: pg leq 2b+2}
    p_g \leq 2b+2.
\end{equation}
It follows that
\[
b\geq p_g -b-2\geq p_g-q-2= \chi -3 \geq 2
\]
where the last inequality is because $\chi\geq 5$ by assumption. 

Now $T$ is either a rational surface or a cone over an elliptic curve by Theorem~\ref{thm: Xiao canonical map} (ii). This implies that the smooth models of $T$ (and of $S/G$) have irregularity at most 1. Thus, the dimension of $H^0(S,\Omega_S)^G$ is at most 1. On the other hand, since $S/\langle\tau\rangle$ is a $\PP^1$-fibration over $B$, we have 
\[
\dim H^0(S,\Omega_S^1)^\tau = b\geq 2.
\]
It follows that $H^0(S,\Omega_S)^G\subsetneq H^0(S,\Omega_S^1)^\tau$, and hence $\sigma$ does not act trivially on $H^0(S, \Omega_S^1)\cong H^1(S, \sO_S)^{\vee}$. Therefore, the induced automorphism $\sigma_B\in \Aut(B)$ of $\sigma$ is not the identity. This finishes the proof of (iii).

\medskip

(iv) we have $\sigma\notin\Aut_\sO(S)$ by (iii), and hence $\Aut_{\sO}(S)$ is trivial. 

\medskip

(v) Note that $q_f\leq 2$ in any case. If $q_f =2$, then the assertion follows directly from Corollary~\ref{cor: q_f max}. 

Thus we may assume that $q_f=1$, and there is a one-form $\eta_0\in H^0(S, \Omega_S^1)\setminus f^*H^0(B, \Omega_B)$, which we may assume to be an eigenvector with eigenvalue $\lambda$ under the action of $\sigma$. Then 
\[
H^0(S,\Omega_S^1)=f^*H^0(\Omega_B^1)\oplus\CC\cdot\eta_0,
\]
both summands of which are invariant under the action of $G=\langle \sigma, \tau\rangle$. Define a $\CC$-linear map
\begin{equation}
\wedge\eta_0\colon H^0(B, \Omega_B^1) \rightarrow H^0(S, K_S),\quad \eta\mapsto f^*\eta\wedge\eta_0
\end{equation}
which is easily seen to be injective and $G$-equivariant. Since $\sigma$ acts trivially on $H^0(S, K_S)$ and hence also trivially on the subspace $f^*H^0(B, \Omega_B^1))\wedge \eta_0$. It follows that $\sigma$ acts as the scalar multiplication by  $\lambda^{-1}$ on the whole $H^0(B,\Omega^1_B)$. Thus the canonical map $\phi_{K_B}\colon B\to\PP^{b-1}$ factors through the quotient map $B\to B/\langle\sigma_B\rangle$. Since $b\geq 2$, we infer that $B$ is a hyperelliptic curve and $\sigma_B$ is its hyperelliptic involution, acting as $-1$ on $H^0(B, \Omega_B^1)$.  It follows that $\lambda =-1$, and the eigenvalues of the action are as in the following table
\[
\begin{array}{c|c|c}
 & f^*H^0(\Omega_B^1) & \CC\cdot\omega \\
\hline
\sigma & -1  & -1 \\
\hline
\tau & 1 & -1
\end{array}
\]
It follows that 
\[
p_g(S/\langle\sigma\tau\rangle)=0,\quad  q(S/\langle\sigma\tau\rangle)=1
\]
and hence $S/\langle\sigma\tau\rangle$ is not of general type (\cite[Chapter VI]{Bea96}). By Corollary~\ref{cor: Z22 acts}, $\sigma$ acts trivially on $\CH_0(S)_{\alb}$.
\end{proof}

Finally, we make an observation that Conjecture~\ref{conj: pg=0} (for surfaces with $p_g=0$) implies Conjecture~\ref{conj: generalize Bloch} for surfaces of general type with a genus two fibration and $\chi(\sO_S)\geq 5$.
\begin{prop}\label{prop: assume conj}
    Let $S$ be a smooth projective surface of general type with a genus two fibration and $\chi(\sO_S)\geq 5$. Assume that Conjecture~\ref{conj: pg=0} holds. Then $\Aut_s(S)$ acts trivially on $\CH_0(S)_{\alb}$.
\end{prop}
\begin{proof}
We may assume that $\Aut_s(S)$ is nontrivial. By Theorem~\ref{thm: main}, we may assume that the canonical $\phi_{K_S}$ is generically finite, and $\Aut_s(S)$ has order two. Let $\sigma$ be the generator of $\Aut_s(S)$, and $\tau\in \Aut_B(S)$ the hyperelliptic involution. Then $\langle\sigma,\tau\rangle\cong (\ZZ/2\ZZ)^2$, and since $S/\langle\tau\rangle\rightarrow B$ is a $\PP^1$-fibration, one has $\CH_0(S/\langle\tau\rangle)_\alb=0$. Note also that $p_g(S/\langle \sigma\tau\rangle)=0$. 

Assuming Conjecture~\ref{conj: pg=0} holds, we have 
$\CH_0(S/\langle\sigma\tau\rangle)_{\alb} = 0$, and hence $\sigma$ acts trivially on $\CH_0(S)_{\alb}$ by Lemma~\ref{lem: Z22 acts}.
\end{proof}


\begin{thebibliography}{99}




	
	\bibitem[Bea82]{Bea82} Arnaud Beauville, L'in\'egalit\'e $p_g\geqslant 2q-4$ pour les surfaces de type g\'en\'eral, Appendice a Olivier Debarre:``In\'egalit\'es  num\'eriques pour les surfaces de type g\'en\'eral",  Bull. Soc. Math. France  110 (1982), no.~3, 319--346.

    \bibitem[Bea96]{Bea96} Arnaud Beauville, Complex algebraic surfaces, London Mathematical Society Student Texts, 34, Cambridge Univ. Press, Cambridge, 1996.

    \bibitem[Blo75]{Blo75} Spencer Bloch, $K_2$ of artinian $\QQ$-algebras, with application to algebraic cycles.  Communications in Algebra, 3 (5), 405--428. https://doi.org/10.1080/00927877508822053
	
	
	
	
	\bibitem[BKL76]{BKL76} Spencer Bloch, Arnold Kas and David Lieberman, Zero cycles on surfaces with $p_g=0$, Compositio Mathematica  33.2 (1976), 135--145.
	
	\bibitem[Cai06a]{Cai06a} Jin-Xing Cai, Automorphisms of fiber surfaces of genus $2$, inducing the identity  in cohomology, Trans. of the Amer. Math. Soc.  358.3 (2005), 1187--1201.
    \bibitem[Cai06b]{Cai06b} Jin-Xing Cai, Automorphisms of fiber surfaces of genus 2, inducing the identity in cohomology. II, Internat. J. Math. 17:2 (2006), 183--193.
	
	\bibitem[Cai07]{Cai07} Jin-Xing Cai, Classification of fiber surfaces of genus $2$ with automorphisms acting trivially in cohomology, Pacific Journal of Mathematics  232.1 (2007), 43--59.

    
	
	\bibitem[DL23]{DL23} Jiabin Du and Wenfei Liu, On symplectic automorphisms of elliptic surfaces acting on $\CH_0$,  Sci. China Math.  66 (2023), no. 3, 443--456. 
	
   
	
	
	
	\bibitem[Huy12]{Huy12} Daniel Huybrechts, Symplectic automorphisms of $K3$ surfaces of arbitrary finite order,  Math. Res. Lett.  19 (2012), no.4, 947--951.
	
	
	

    \bibitem[JLZ23]{JLZ23} Zhi Jiang, Wenfei Liu and Hang Zhao, On numerically trivial automorphisms of threefolds of general type.  Math. Res. Lett.  30 (2023), no. 6, 1751--1785.
	
	\bibitem[Kim05]{Kim05} Shun-ichi Kimura, Chow groups are finite dimensional, in some sense,  Math. Ann.  331 (2005), 173--201.
	

    \bibitem[Lat21]{Lat21}  Robert Laterveer, Bloch's conjecture for some numerical Campedelli surfaces,  Asian Journal of Mathematics, Vol. 25, No. 1, pp. 049--064, 2021.

    \bibitem[LYZ23]{LYZ23} Zhiyuan Li, Xun Yu and Ruxuan Zhang, Bloch's conjecture for (anti-)autoequivalences on K3 surfaces, arXiv:2305.10078v2, 2023.
 
	\bibitem[Mum69]{Mum69} David Mumford, Rational equivalence of $0$-cycles on surfaces, J. Math. Kyoto Univ.  9 (1969), 195--204.
	
	
	
	
	
	\bibitem[Paw19]{Paw19} Rakesh R. Pawar, Action of correspondences on filtrations on cohomology and $0$-cycles of Abelian varieties,  Mathematische Zeitschrift  292: 655--675, 2019.

    \bibitem[Roj80]{Roj80} A. A. Rojtman, The torsion of the group of $0$-cycles modulo rational equivalence. Ann. of Math.  (2) 111 (1980), no. 3, 553--569. 

	
	
	
	
	
	
	
	
	
	
	\bibitem[Voi03]{Voi03} Claire Voisin,  Hodge Theory and Complex Algebraic Geometry II, Cambridge studies in advanced mathematics 77, Cambridge University Press, 2003.

    \bibitem[Voi12]{Voi12} Claire Voisin, Symplectic involutions of $K3$ surfaces act trivially on $\CH_0$,  Documenta Mathematica  17 (2012), 851--860.

	
	\bibitem[Xiao85]{Xiao85} Gang Xiao, Surfaces fiberes en courbes de genre deux,  Lecture Notes in Mathematics 1137, Springer-Verlag Berlin Heidelberg, 1985.
 
 
	
\end{thebibliography}
\end{document}